\documentclass[psamsfonts]{amsart}

\usepackage[T1]{fontenc}
\usepackage{amssymb}

\usepackage[dvips]{graphicx}

\usepackage{graphicx}
\usepackage{amssymb}                       
\usepackage{amsmath}
\usepackage{color}
\usepackage{times}

\usepackage[unicode,bookmarks,colorlinks]{hyperref}
\hypersetup{
   linkcolor=brickred,
}

\definecolor{mahogany}{cmyk}{0, 0.77, 0.87, 0}
\definecolor{salmon}{cmyk}{0, 0.53, 0.38, 0}
\definecolor{melon}{cmyk}{0, 0.46, 0.50, 0}
\definecolor{yellowgreen}{cmyk}{0.44, 0, 0.74, 0}
\definecolor{brickred}{cmyk}{0, 0.89, 0.94, 0.28}
\definecolor{OliveGreen}{cmyk}{0.64, 0, 0.95, 0.40}
\definecolor{RawSienna}{cmyk}{0, 0.72, 1.0, 0.45}
\definecolor{ZurichRed}{rgb}{1, 0, 0} 

\usepackage{fancyhdr}
\usepackage{amsmath,amstext,amssymb,amsopn,amsthm}
\usepackage{amsmath,amssymb,amsthm}
\usepackage[mathscr]{eucal}

\reversemarginpar

\definecolor{rb}{rgb}{0.1,0.2, 0.7}

\begin{document}

\newtheorem{lemma}[thm]{Lemma}
\newtheorem{proposition}{Proposition}
\newtheorem{theorem}{Theorem}[section]
\newtheorem{deff}[thm]{Definition}
\newtheorem{case}[thm]{Case}
\newtheorem{prop}[thm]{Proposition}
\newtheorem{example}{Example}

\newtheorem{corollary}{Corollary}

\theoremstyle{definition}
\newtheorem{remark}{Remark}

\numberwithin{equation}{section}
\numberwithin{definition}{section}
\numberwithin{corollary}{section}

\numberwithin{theorem}{section}

\numberwithin{remark}{section}
\numberwithin{example}{section}
\numberwithin{proposition}{section}

\newcommand{\gap}{\lambda_{2,D}^V-\lambda_{1,D}^V}
\newcommand{\gapR}{\lambda_{2,R}-\lambda_{1,R}}
\newcommand{\bD}{\mathrm{I\! D\!}}
\newcommand{\calD}{\mathcal{D}}
\newcommand{\calA}{\mathcal{A}}

\newcommand{\conjugate}[1]{\overline{#1}}
\newcommand{\abs}[1]{\left| #1 \right|}
\newcommand{\cl}[1]{\overline{#1}}
\newcommand{\expr}[1]{\left( #1 \right)}
\newcommand{\set}[1]{\left\{ #1 \right\}}

\newcommand{\calC}{\mathcal{C}}
\newcommand{\calE}{\mathcal{E}}
\newcommand{\calF}{\mathcal{F}}
\newcommand{\Rd}{\mathbb{R}^d}
\newcommand{\BR}{\mathcal{B}(\Rd)}
\newcommand{\R}{\mathbb{R}}
\newcommand{\T}{\mathbb{T}}
\newcommand{\D}{\mathbb{D}}

\newcommand{\al}{\alpha}
\newcommand{\RR}[1]{\mathbb{#1}}
\newcommand{\bR}{\mathrm{I\! R\!}}
\newcommand{\ga}{\gamma}
\newcommand{\om}{\omega}
\newcommand{\A}{\mathbb{A}}
\newcommand{\bH}{\mathbb{H}}

\newcommand{\bb}[1]{\mathbb{#1}}
\newcommand{\bI}{\bb{I}}
\newcommand{\bN}{\bb{N}}

\newcommand{\uS}{\mathbb{S}}
\newcommand{\M}{{\mathcal{M}}}
\newcommand{\calB}{{\mathcal{B}}}

\newcommand{\W}{{\mathcal{W}}}

\newcommand{\m}{{\mathcal{m}}}

\newcommand {\mac}[1] { \mathbb{#1} }

\newcommand{\bC}{\Bbb C}

\newtheorem{rem}[theorem]{Remark}
\newtheorem{dfn}[theorem]{Definition}
\theoremstyle{definition}
\newtheorem{ex}[theorem]{Example}
\numberwithin{equation}{section}

\newcommand{\Pro}{\mathbb{P}}
\newcommand\F{\mathcal{F}}
\newcommand\E{\mathbb{E}}
\newcommand\e{\varepsilon}
\def\H{\mathcal{H}}
\def\t{\tau}

\title[Maximal inequalities]{A sharp two-weight inequality\\ for fractional maximal operators}

\author{Rodrigo Ba\~nuelos}\thanks{R.~Ba\~nuelos was supported in part by NSF Grant 1854709-DMS}
\address{Department of Mathematics, Purdue University, West Lafayette, IN 47907, USA}
\email{banuelos@math.purdue.edu}

\author{Adam Os\k ekowski}
\address{Department of Mathematics, Informatics and Mechanics\\
 University of Warsaw\\
Banacha 2, 02-097 Warsaw\\
Poland}
\email{ados@mimuw.edu.pl}

\thanks{}

\begin{abstract}
The paper is devoted to two-weight estimates for the fractional maximal operators $\mathcal{M}^\alpha$ on general probability spaces equipped with a tree-like structure. For given $1<p\leq q<\infty$, we study the sharp universal upper bound for the norm 
$ \|\mathcal{M}^\alpha\|_{L^p(v)\to L^q(u)}$, where $(u,v)$ is an arbitrary pair of weights satisfying the Sawyer testing condition. The proof is based on the abstract Bellman function method, which reveals an unexpected connection of the above problem with the sharp version of the classical Sobolev imbedding theorem. 
\end{abstract}

\maketitle

\section{Introduction}
Our motivation comes from the questions concerning the best constants in weighted $L^p\to L^q$ inequalities  for fractional maximal operators on $\R^d$. To present the results in an appropriate context, let us start with the necessary background and notation. For a fixed constant $0\leq \alpha<d$, the fractional maximal operator $\mathcal{M}^\alpha$ acts on locally integrable functions $\varphi$ on $\R^d$ by the formula
$$ \mathcal{M}^\alpha\varphi(x)=\sup\left\{|Q|^{\alpha/d}\langle |\varphi|\rangle_Q:Q\subset \R^d\mbox{ is a cube containing }x\right\}.$$
Here $|Q|$ denotes the Lebesgue measure of $Q$, $\langle \psi\rangle_Q=\frac{1}{|Q|}\int_Q \psi$ stands for the average of $\psi$ over $Q$, and the cubes we consider in the above supremum have  sides parallel to the coordinate axes. In particular, if we put $\alpha=0$, then $\mathcal{M}^\alpha$ reduces to the classical Hardy-Littlewood maximal operator.  

A closely related object, the so-called Riesz potential (fractional integral operator) $I_\alpha$, is given by the convolution
$$ I_\alpha \varphi(x)=\int_{\R^d} \frac{\varphi(y)}{|x-y|^{d-\alpha}}\mbox{d}y,\qquad 0<\alpha<d.$$
These fractional operators play an important role in analysis and PDEs, as they are a convenient tool to study differentiability properties of functions. Via the Fourier transform or heat semigroup they represent negative powers of the Laplacian. There is a huge literature on the boundedness of these operators in various function spaces and geometric settings. Such bounds often go under the name of the Hardy-Littlewood-Sobolev (HLS) inequality.  It has been of great interest to obtain optimal, or at least good bounds, for the corresponding norms. For a sample of this literature, see \cite{Christ,Esc, Frank1, Frank2,Gr,He,L,LMS,Lieb,St, Varo}.

In this paper, we will study the boundedness of fractional maximal functions on weighted $L^p$ spaces. In what follows, the  word ``weight'' refers to a locally integrable, positive function  
on $\R^d$, which will usually be denoted by $w$. With a slight abuse of notation, we will also write $w(A)=\int_A w\mbox{d}x$ for any measurable subset $A$ of $\R^d$. Given $p\in (0,\infty)$, the associated weighted $L^p$ space is given by
$$ L^p(w)=\left\{f:\R^d\to \R\,:\,\|f\|_{L^p(w)}:=\left(\int_{\R^d}|f|^pw\right)^{1/p}<\infty\right\}.$$
\def\notneeded{Now suppose that $p\in (1,\infty)$. We say that $w$ belongs to the Muckenhoupt class $A_p$ (or, in short, that $w$ is an $A_p$ weight), if the $A_p$ characteristics $[w]_{A_p}$, given by
$$ [w]_{A_p}:=\sup_Q \left(\frac{1}{|Q|}\int_Q w \right)\left(\frac{1}{|Q|}\int_Q w^{-1/(p-1)}\right)^{p-1},$$
is finite.  This condition arises naturally during the study of weighted $L^p$ inequalities for the Hardy-Littlewood maximal operator, see Muckenhoupt \cite{M}. To be more precise, for a given $1<p<\infty$, the inequality 
$$ \|\mathcal{M}^0\varphi\|_{L^p(w)}\leq C_{p,d,w}||\varphi||_{L^p(w)}$$
holds true with some constant $C_{p,d,w}$ depending only on the parameters indicated, if and only if $w$ is an $A_p$ weight. Here 
$ \|\varphi\|_{L^p(w)}=\left(\int_{\R^d} |\varphi|^pw\mbox{d}u\right)^{1/p}$ is the usual norm in the weighted space $L^p(w)$. This result was extended to the fractional setting by Muckenhoupt and Wheeden \cite{MW}. Let $p,\,q$ be positive exponents satisfying the relation $1/q=1/p-\alpha/d$. Then the inequality
$$ \left(\int_{\R^d} \big(\mathcal{M}^\alpha \varphi(x)\big)^qw(x)^q \mbox{d}x\right)^{1/q}\leq C_{p,\alpha,d,w}\left(\int_{\R^d} |\varphi(x)|^pw(x)^p\mbox{d}x\right)^{1/p}$$
holds true if and only if 
$$ \sup_Q \left(\frac{1}{|Q|}\int_Q w^q\right)\left(\int_Q w^{-p'}\right)^{q/p'}<\infty,$$
where $p'=p/(p-1)$ is the harmonic conjugate to $p$. In other words, we have
$$ \|\mathcal{M}^\alpha\|_{L^p(w^{p/q})\to L^q(w)}\leq C_{p,\alpha,d,w}<\infty$$
 if and only if $w\in A_{q/p'+1}.$ }
 
We will be mostly concerned with the two-weight context. Fix $\alpha\in [0,d)$ and two parameters $1<p\leq q<\infty$.  The problem is to characterize those  pairs $(u,v)$ of weights on $\R^d$ for which $\mathcal{M}^\alpha$ is bounded as an operator from $L^p(v)$ to $L^q(u)$. When $u=v$ and $p=q$, this reduces to the classical one-weight problem solved by Muckenhoupt \cite{M} in the seventies, where the characterization is expressed in terms of the so-called $A_p$ condition. The two-weight case was successfully handled by Sawyer \cite{Sa}. Namely, we have that $\|\mathcal{M}^\alpha\|_{L^p(v)\to L^q(u)}<\infty$ if and only if the following testing condition holds: there is a finite universal constant $C_{test}$ such that for all cubes $Q$,
\begin{equation}\label{testt}
 \left(\int_Q \big(\mathcal{M}^\alpha(v^{1-p'}\chi_Q)\big)^qu\mbox{d}x\right)^{1/q}\leq C_{test} \left(\int_Q v^{1-p'}\mbox{d}x\right)^{1/p}.
\end{equation}
Here $p'=p/(p-1)$ stands for the harmonic conjugate to $p$. 
There is an analogue of the above result in the dyadic case, or more generally, in the context of probability spaces equipped with a tree-like structure. Let us provide the description of this more general setup.

\begin{dfn}\label{tree}
Suppose that  $(X,\mu)$ is a probability space. 
A set $\mathcal{T}$ of measurable subsets of $X$ will be called a tree if the following conditions are satisfied:
\begin{itemize}
\item[(i)] $X\in \mathcal{T}$ and for every $Q\in\mathcal{T}$ we have $\mu(Q)>0$.

\smallskip

\item[(ii)] For every $Q \in \mathcal{T}$ there is a  finite subset $C(Q) \subset \mathcal{T}$ containing at least two elements such that

(a) the elements of $C(Q)$ are pairwise disjoint subsets of $Q$,

(b) $ Q = \bigcup C(Q)$.

\smallskip

\item[(iii)] $\mathcal{T} = \bigcup_{m\geq 0} \mathcal{T}^m$, where $\mathcal{T}^0=\{X\}$ and $\mathcal T^{m+1} = \bigcup_{Q\in\mathcal{T}^m} C(Q)$.
\end{itemize}
\end{dfn}

An important comment is in order. Typically, one assumes that the space $(X,\mu)$ is nonatomic and $\mathcal{T}$ enjoys the additional property

\smallskip

\begin{itemize}
\item[(iv)] We have $\lim_{m\to\infty}\sup_{Q\in\mathcal{T}^m} \mu(Q) = 0$.
\end{itemize}

\smallskip

\noindent This extra condition enables the use of differentiation theorems. However, we will work with the general context described by (i)-(iii) only, as the validity of (iv) may affect the two-weight estimates for some values of parameters. See the discussion following Theorem \ref{mainthm} below.

Any probability space equipped with a tree gives rise to the corresponding fractional maximal operator $\mathcal{M}^\alpha_\mathcal{T}$, acting on integrable functions $f:X\to \R$ by the formula
$$ \mathcal{M}^\alpha_\mathcal{T}f(x)=\sup\left\{\mu(Q)^{\alpha}\langle |f|\rangle_{Q}: x\in Q, \,Q \in\mathcal{T}\right\}.$$
Here $\langle \psi\rangle_{Q}=\frac{1}{\mu(Q)}\int_Q \psi\mbox{d}\mu$ stands for the average of $\psi$ over $Q$ with respect to the underlying probability measure $\mu$. 
Note that in contrast to the Euclidean case discussed above, there is no dimension here and hence the operator depends solely on $\alpha$. It is easy to check that the definition of $\mathcal{M}^\alpha_\mathcal{T}$ is meaningful in the range $0\leq \alpha<1$ only, and the case $\alpha=0$ corresponds to the standard (martingale) maximal function.

The concept of a weight and related notions carry over effortlessly: a (probabilistic) weight is a positive, integrable random variable $w$, we will use the same letter to denote the induced measure on $X$. Any weight $w$ gives rise to the associated weighted $L^p$ space, which will again be denoted by $L^p(w)$. 

Now, pick $(X,\mu)$ and $\mathcal{T}$, and let $\alpha\in [0,1)$, $1<p\leq q<\infty$ be fixed. In analogy to the previous context, one may ask about the characterization of those pairs $(u,v)$ of probabilistic weights, for which the maximal function $\mathcal{M}^\alpha_\mathcal{T}$ is bounded as an operator from $L^p(v)$ to $L^q(u)$. It follows from the work of Sawyer \cite{Sa} that the boundedness holds if and only if we have the testing condition
$$ \left(\int_Q \big(\mathcal{M}^\alpha_\mathcal{T}(v^{1-p'}\chi_Q)\big)^qu\mbox{d}\mu\right)^{1/q}\leq C_{test} \left(\int_Q v^{1-p'}\mbox{d}\mu\right)^{1/p}\quad \mbox{ for all }Q\in \mathcal{T},$$
for some universal $C_{test}<\infty$. 
One part of the implication is trivial: if we have the boundedness, then for any $Q\in \mathcal{T}$ we have
\begin{align*}
 \left(\int_Q \big(\mathcal{M}^\alpha_\mathcal{T}(v^{1-p'}\chi_Q)\big)^qu\mbox{d}\mu\right)^{1/q}\leq \|\mathcal{M}^\alpha_\mathcal{T} (v^{1-p'}\chi_Q)\|_{L^q(u)}\\
 \leq \|\mathcal{M}^\alpha_\mathcal{T}\|_{L^p(v)\to L^q(u)}\|v^{1-p'}\chi_Q\|_{L^p(v)},
 \end{align*}
which is the testing condition with $C_{test}= \|\mathcal{M}_\mathcal{T}^\alpha\|_{L^p(v)\to L^q(u)}.$ We will be interested in the reverse implication. For  given $1<p\leq q<\infty$, distinguish the constant
\begin{equation}\label{defC}
 C_{p,q}=\left(p-1\right)^{-1/q}\left[\frac{\Gamma(pq/(q-p))}{\Gamma(q/(q-p))\Gamma(p(q-1)/(q-p))}\right]^{1/p-1/q},
\end{equation}
where $\Gamma$ is the Euler Gamma function. 
If $p=q$, the above number is understood by passing to the limit: $C_{p,p}=p/(p-1)$. 

\begin{theorem}\label{mainthm}
Fix $1<p\leq q<\infty$, $0\leq \alpha<1$ and $L>0$. Suppose that $u$, $v$ are two weights on $(X,\mu)$ such that
\begin{equation}\label{testing}
 \left(\int_Q \mathcal{M}^\alpha_\mathcal{T}(\chi_Q v^{1-p'})^q u\mbox{d}\mu\right)^{1/q}\leq L\left(\int_Q v^{1-p'}\mbox{d}\mu\right)^{1/p}
\end{equation}
for any $Q\in \mathcal{T}$. Then we have the estimate
\begin{equation}\label{mainin}
 \|\mathcal{M}^\alpha_\mathcal{T}\|_{L^p(v)\to L^q(u)}\leq C_{p,q}L.
 \end{equation}
If $\alpha\geq 1/p-1/q$, then the constant $C_{p,q}$ is the best possible: for any $1<p\leq q<\infty$ and any $c<C_{p,q}L$ there is a triple $(X,\mu,\mathcal{T})$ and a pair $(u,v)$ of weights satisfying the testing condition \eqref{testing}, for which $ \|\mathcal{M}^\alpha_\mathcal{T}\|_{L^p(v)\to L^q(u)}>c.$
\end{theorem}
We may and do assume that the testing constant $L$ appearing in \eqref{testing} is equal to $1$: this does not affect the generality of the result, since the constant can always be absorbed into the weight $u$. 

Unfortunately, we have been unable to provide the sharp result for all $\alpha$: while the estimate \eqref{mainin} holds true in the full range, it is not clear to us whether $C_{p,q}$ is the best in the case $\alpha<1/p-1/q$. However, we would like to point out that the range $\alpha\in [1/p-1/q,1)$, in which we provide the full answer, is a natural range of parameters, at least for ``friendly'' probability spaces. Namely, suppose that $\alpha<1/p-1/q$ and assume that the probability space $(X,\mu)$ is nonatomic and satisfies the differentiability property (iv) discussed above. Then there is no chance for the $L^p(v)\to L^q(u)$ estimate, unless the weight $u$ vanishes almost everywhere. Indeed, setting $f=\chi_Q$ for some $Q\in \mathcal{T}$, we see that
$$ \|\mathcal{M}^\alpha_\mathcal{T}\|_{L^p(v)\to L^q(u)}\geq \frac{\|\mathcal{M}^\alpha_\mathcal{T} f\|_{L^q(u)}}{\|f\|_{L^p(v)}}\geq \frac{\mu(Q)^{\alpha}(u(Q))^{1/q}}{(v(Q))^{1/p}}=\mu(Q)^{\alpha+1/q-1/p}\cdot \frac{\langle u\rangle_Q^{1/q}}{\langle v\rangle_Q^{1/p}}.$$
However, by Lebesgue's differentiation theorem (or rather Doob's martingale convergence theorem), for almost all $\omega\in X$, the averages $\langle u\rangle_Q$ and $\langle v\rangle_Q$ tend to $u(\omega)$ and $v(\omega)$ as $Q$ shrinks to $\omega$. Hence if $u>0$ with positive probability, then the expression on the right above explodes for an appropriate sequence of $Q$'s with $\mu(Q)\to 0$, and hence $\mathcal{M}^\alpha_\mathcal{T}$ does not map $L^p(v)$ to $L^q(u)$ boundedly. 

 Our approach will rest on the so-called Bellman function technique, a powerful method used widely in probability and harmonic analysis to obtain tight estimates. The technique originates from the theory of optimal stochastic control developed by Bellman \cite{Be}, and it has been studied intensively during the last thirty years. Its connection to the problems of martingale theory was firstly observed by Burkholder \cite{Bur1}, who used it to identify the unconditional constant of the Haar system and related estimates for martingale transforms (some echoes of this approach can be found in his earlier paper \cite{B0} on geometric characterization of UMD spaces). This direction of research was further explored intensively by Burkholder, his PhD students and other mathematicians (see \cite{Ose} for the overview). In the nineties, Nazarov, Treil and Volberg (cf. \cite{NT}, \cite{NTV}) described the method from a wider perspective which allowed them to apply it to various problems of harmonic analysis. Since then, the technique has proven to be extremely efficient in various contexts.  Consult for example  \cite{BO,Os,P,SV,SV2,V,VV}, and the many references therein.

Roughly speaking, the technique states that the validity of a given estimate is equivalent to the existence of a certain special function, enjoying appropriate size and concavity requirements; the shape of these conditions depend on the context and the form of the inequality under investigation. In a typical situation, the identification of the appropriate Bellman function might be a very difficult task. One searches for the function with experimentation, often using arguments coming from various areas of mathematics  such as PDEs, differential geometry, probability theory and calculus of variation. In general, the candidate obtained is extremely complicated, especially in the case when one aims at the best constant (often the function involves three or more variables), and the verification of the relevant properties is very elaborate. Our approach in this paper will be completely different, and it will allow us to avoid most of the technical problems. Our starting point is the sharp version of the classical Sobolev imbedding theorem, established by Talenti \cite{T} in the seventies, or rather its variant for radial functions, due to Bliss \cite{Bl}. Using the other implication of the Bellman function method, the validity of this inequality implies the existence of the corresponding special function. We then show that this abstract, non-explicit function, enjoys all the properties and necessary conditions to yield \eqref{mainin}. This unexpected relation between the sharp Sobolev imbedding theorem and sharp two-weight inequalities for fractional maximal operators can be regarded as an independent contribution of this work. This gives rise to the very interesting and challenging general problem of formulating similar connections in other geometric settings. 

The remaining part of the paper is organized as follows. In the next section we show how the Sobolev imbedding theorem leads to a certain special function of four variables, and we study some additional properties of this object. In the final part of the paper we exploit this function to deduce the desired estimate \eqref{mainin}; furthermore, we construct there examples showing that the constant $C_{p,q}$ is the best possible for $\alpha\geq 1/p-1/q$.

\section{Sobolev imbedding theorem, Bliss' inequality and a special function}

Suppose that $d$ is a fixed dimension, $p$ is an exponent lying in the interval $(1,d)$ and let $q=dp/(d-p)$. The classical Sobolev inequality asserts the existence of the constant $K=K_{p,d}$ depending only on the parameters indicated such that
\begin{equation}\label{sobolev}
 \|u\|_{L^q(\R^d)}\leq K_{p,d} \|\,|\nabla u|\,\|_{L^p(\R^d)}
\end{equation}
for any smooth compactly supported function $u$ on $\R^d$. In  \cite{T} Talenti identified the best value of the constant $K_{p,d}$ to be 
$$ \frac{1}{\pi^{1/2}d^{1/p}} \left(\frac{p-1}{d-p}\right)^{1-1/p}\left(\frac{\Gamma(1+d/2)\Gamma(d)}{\Gamma(d/p)\Gamma(1+d-d/p)}\right) ^{1/d}.$$
Specifically, using symmetrization Talenti proved that it is enough to establish the inequality for radial functions: $u(x)=F(|x|)$ for some smooth compactly supported function $F:[0,\infty)\to \R$. If we write $ F(t)=-\int_t^\infty F'(s)\mbox{d}s$ and substitute $f=-F'$, 
then $\nabla u(x)=-f(|x|)x/|x|$ and \eqref{sobolev} becomes
\begin{equation}\label{talenti}
\begin{split}
 &\left(\int_0^\infty \left|\int_t^\infty f(s)\mbox{d}s\right|^qt^{d-1}\mbox{d}t\right)^{1/q}\\
&\leq K_{p,d}\left(\frac{d\pi^{d/2}}{\Gamma(1+d/2)}\right)^{1/p-1/q}\left(\int_0^\infty |f(t)|^pt^{d-1}\mbox{d}t\right)^{1/p}.
\end{split}
\end{equation}
Talenti showed this estimate using techniques from the calculus of variations. However, it is worth mentioning  that the bound had already appeared in the literature. If we set $\alpha=q(p-1)/(pd)$ and make the substitution $t=u^{-\alpha}$ under both integrals, and then put $\varphi(u)=f(u^{\alpha})u^{-\alpha-1}$, then \eqref{talenti} transforms into
\begin{equation}\label{bliss}
 \left(\int_0^\infty u^{q/p-1}\left|\frac{1}{u}\int_0^u \varphi\right|^q\mbox{d}u\right)^{1/q}\leq \left(\frac{p}{q}\right)^{1/q}C_{p,q}\left(\int_0^\infty |\varphi|^p\right)^{1/p},
\end{equation}
where $C_{p,q}$ is given by \eqref{defC}. 
This is a classical inequality due to Bliss \cite{Bl}, who proved it in 1930 also with calculus of variations. There is an alternative proof of \eqref{bliss}, invented in \cite{Os}, which makes use of Bellman functions. Let us briefly discuss this approach, as some elements will be useful to us later. 
The estimate \eqref{bliss} is equivalent to proving that 
$$ \sup\left\{\int_0^s u^{q/p-1}\left(\frac{1}{u}\int_0^u \varphi\right)^q\mbox{d}u\right\}=\frac{p}{q}C_{p,q}^q,$$
where the supremum is taken over all $s$ and all nonnegative measurable functions $\varphi$ on $(0,s)$ satisfying $\int_0^s \varphi^p=1$. We extend this problem to the following: given $s>0$ and $x,\,y\geq 0$ satisfying $x^p\leq y$, consider the quantity
$$ b(x,y,s)=\sup\left\{\int_0^s u^{q/p-1}\left(\frac{1}{u}\int_0^u \varphi\right)^q\mbox{d}u\right\},$$
where this time the supremum is taken over all $\varphi:(0,s)\to [0,\infty)$ satisfying $\frac{1}{s}\int_0^s \varphi=x$ and $\frac{1}{s}\int_0^s \varphi^p=y.$ Note that the assumption $x^p\leq y$ guarantees the existence of a function $\varphi$ satisfying these integral properties, and hence the definition of $b(x,y,s)$ makes sense. At a first glance this might be confusing: we have increased the difficulty of the problem, adding some extra bounds on the function $\varphi$. However, it turns out that the introduction of the additional variables enables the explicit derivation of the function $b$, with the help of optimal control theory. More specifically, this function can be shown to satisfy a certain partial differential equation (the associated Hamilton-Jacobi-Bellman equation), together with some homogeneity-type properties, and the computation can be carried out. Then one checks that $b(x,y,s)\leq \frac{p}{q}C_{p,q}^q(ys)^{q/p}$, which yields the estimate \eqref{bliss}. See \cite{Os} for details.

The successful treatment of our main result requires the further complication of the Bellman function $b$ introduced above. Let us distinguish the four-dimensional domain
$$ \mathcal{D}=\Big\{(x,y,s,t)\in \R_+^4\,:\,x^p\leq y,\,s\geq t^{p/q}\Big\}.$$
Let $\mathcal{B}:\mathcal{D}\to \R$ be given by
\begin{equation}\label{defB}
 \mathcal{B}(x,y,s,t)=\sup\left\{\int_0^{t^{p/q}} u^{q/p-1}\left(\frac{1}{u}\int_0^u \varphi\right)^q\mbox{d}u\right\},
\end{equation}
where the supremum is taken over all measurable functions $\varphi:[0,s]\to [0,\infty)$, satisfying
$$ \frac{1}{s}\int_0^s \varphi=x\qquad \mbox{ and }\qquad \frac{1}{s}\int_0^s \varphi^p=y.$$
Thanks to the condition $s\geq t^{p/q}$, the integrals under the supremum are well-defined and hence the definition of $\mathcal{B}$ makes sense. Note that $\mathcal{B}$ is indeed a complication of $b$. That is, we have the identity $b(x,y,s)=\mathcal{B}(x,y,s,s^{q/p})$, and the extra variable $t$ controls the length of the interval of integration. 

The function $\mathcal{B}$, or more precisely certain pointwise estimates satisfied by it, will be of key importance for the proof of \eqref{mainin}. Here is the main result of this section.

\begin{theorem}\label{PropB}
The function $\mathcal{B}$ enjoys the following three properties.

\smallskip

1$^\circ$ $\mathcal{B}\geq 0$.

2$^\circ$ For any $(x,y,s,t)\in \mathcal{D}$, $\mathcal{B}(x,y,s,t)\leq \frac{p}{q}C_{p,q}^q (sy)^{q/p}$.

3$^\circ$ For any $m=2,\,3,\,\ldots$, the estimate
$$\mathcal{B}(x,y,s,t)\geq \frac{p}{q}x^qu+\sum_{i=1}^m \mathcal{B}(x_i,y_i,s_i,t_i)$$
holds provided $(x_i,y_i,s_i,t_i)\in \mathcal{D}$, $i=1,\,2,\,\ldots,\,m$, satisfy
\begin{equation}\label{conditions}
 \sum_{i=1}^m s_ix_i=sx,\quad \sum_{i=1}^m s_iy_i=sy,\qquad u+\sum_{i=1}^m t_i=t\quad\mbox{ and }\quad  \sum_{i=1}^m s_i=s.
\end{equation}
\end{theorem}

It is natural to try to prove the above statements as follows: first compute $\mathcal{B}$ explicitly, and then verify the above conditions ``by hand''. The problem with this approach is that the formula for $\mathcal{B}$ is extremely complicated and involves some auxiliary parameters given implicitly (this is also true for $b$, see \cite{Os}). Fortunately, we will be able to prove Theorem \ref{PropB} in a different manner, based only on the abstract definition of $\mathcal{B}$ and thus avoiding most of the technically complicated  issues.

Actually, the main difficulty lies in showing the condition 3$^\circ$. Indeed, by a direct application of \eqref{defB}, we see that the property 1$^\circ$ is trivial, and 2$^\circ$ follows at once from Bliss' inequality \eqref{bliss}. The proof of the third condition follows from the next proposition which is proved in the two lemmas that follow. 

\begin{proposition}
To show 3$^\circ$, it is enough to prove that $\mathcal{B}$ enjoys the following two properties. 

\smallskip
3$^\circ$' If $(x,y,s,t)\in \mathcal{D}$ and $u\in [0,t]$, then
$$ \mathcal{B}(x,y,s,t)\geq \frac{p}{q}x^qu+\mathcal{B}(x,y,s,t-u).$$

\smallskip

3$^\circ$'' We have
$$ \mathcal{B}(x,y,s,t)\geq \mathcal{B}(x_1,y_1,s_1,t_1)+\mathcal{B}(x_2,y_2,s_2,t_2),$$
provided the points $(x,y,s,t)$, $(x_1,y_1,s_1,t_1)$, $(x_2,y_2,s_2,t_2)\in \mathcal{D}$ satisfy
$$ s_1x_1+s_2x_2=sx,\qquad s_1y_1+s_2y_2=sy,\qquad t_1+t_2=t\qquad \mbox{and}\qquad s_1+s_2=s.$$ 
\end{proposition}
\begin{proof} Assume 3$^\circ$' and 3$^\circ$''. 
Pick arbitrary $(x,y,s,t),\,(x_i,y_i,s_i,t_i)$ and $u$ as in the formulation of 3$^\circ$ and let  $t'=\sum t_i=t-u$. The application of 3$^\circ$' yields
\begin{equation}\label{first_reduction}
 \mathcal{B}(x,y,s,t)\geq \frac{p}{q}x^qu+\mathcal{B}(x,y,s,t').
\end{equation}
 However, by induction on $m$, 3$^\circ$'' gives
\begin{equation}\label{induction}
 \mathcal{B}(x,y,s,t')\geq \sum_{i=1}^m \mathcal{B}(x_i,y_i,s_i,t_i).
\end{equation}
Indeed, if $(x_{m-1},y_{m-1},s_{m-1},t_{m-1})$ and $(x_m,y_m,s_m,t_m)$ belong to $\mathcal{D}$, then so does 
$$ \left(\frac{s_{m-1}x_{m-1}+s_mx_m}{s_{m-1}+s_m},\frac{s_{m-1}y_{m-1}+s_my_m}{s_{m-1}+s_m},s_{m-1}+s_m,t_{m-1}+t_m\right),$$
since 
$$ \frac{s_{m-1}y_{m-1}+s_my_m}{s_{m-1}+s_m}\geq \frac{s_{m-1}x_{m-1}^p+s_mx_m^p}{s_{m-1}+s_m}\geq \left(\frac{s_{m-1}x_{m-1}+s_mx_m}{s_{m-1}+s_m}\right)^p,$$
by Jensen's inequality, and
$$ s_{m-1}+s_m\geq t_{m-1}^{p/q}+t_m^{p/q}\geq (t_{m-1}+t_m)^{p/q},$$
because $p\leq q$. This permits the use of induction step and \eqref{induction} follows. Combining it with \eqref{first_reduction} yields  3$^\circ$.
\end{proof}

Now we turn our attention to the conditions 3$^\circ$' and 3$^\circ$''.

\begin{lemma}
The function $\mathcal{B}$ enjoys the condition 3$^\circ$'.
\end{lemma}
\begin{proof}
 Pick an arbitrary $\varphi$ as in the definition of $\mathcal{B}(x,y,s,t-u)$.  Let $\tilde \varphi$ be the decreasing rearrangement of $\varphi$: $\tilde{\varphi}$ has the same distribution as $\varphi$ and is nonincreasing on $[0,s]$. Then $\tilde{\varphi}$ is taken into account when computing $\mathcal{B}(x,y,s,t-u)$. Furthermore, by monotonicity, we have $\frac{1}{w}\int_0^w \tilde \varphi\geq \frac{1}{w}\int_0^w \varphi$ and $\frac{1}{w}\int_0^w \tilde \varphi\geq \frac{1}{s}\int_0^s \tilde \varphi=x$ for all $w$. Therefore,
\begin{align*}
&\mathcal{B}(x,y,s,t)\\
 &\geq \int_0^{t^{p/q}} w^{q/p-1}\left(\frac{1}{w}\int_0^w \tilde \varphi\right)^q\mbox{d}w\\
&=\int_0^{(t-u)^{p/q}} w^{q/p-1}\left(\frac{1}{w}\int_0^w \tilde \varphi\right)^q\mbox{d}w+\int_{(t-u)^{p/q}}^{t^{p/q}} w^{q/p-1}\left(\frac{1}{w}\int_0^w \tilde \varphi\right)^q\mbox{d}w\\
&\geq  \int_0^{(t-u)^{p/q}} w^{q/p-1}\left(\frac{1}{w}\int_0^w \varphi\right)^q\mbox{d}w+\frac{p}{q}x^qu.
\end{align*}
Taking the supremum over $\varphi$, we get 3$^\circ$'.
\end{proof}

\begin{lemma}
The function $\mathcal{B}$ enjoys the condition 3$^\circ$''.
\end{lemma}
\begin{proof}
Here the argument is more complicated. Pick arbitrary $\varphi_1$ and $\varphi_2$ as in the definitions of $\mathcal{B}(x_1,y_1,s_1,t_1)$ and $\mathcal{B}(x_2,y_2,s_2,t_2)$. Concatenate them into a single function $ \varphi$ on $[0,s_1+s_2]$, setting $\varphi(r)=\varphi_1(r)$ if $r\in [0,s_1]$ and $\varphi(r)=\varphi_2(r-s_1)$ if $s\in [s_1,s_1+s_2]$. Next, let $\tilde \varphi:[0,s_1+s_2]\to [0,\infty)$ be the decreasing rearrangement of $ \varphi$. We have
\begin{align*}
 \frac{1}{s_1+s_2}\int_0^{s_1+s_2} \tilde \varphi&=\frac{1}{s_1+s_2}\int_0^{s_1}\varphi_1+\frac{1}{s_1+s_2}\int_0^{s_2}\varphi_2\\
&=\frac{s_1x_1}{s_1+s_2}+\frac{s_2x_2}{s_1+s_2}=x
\end{align*}
and similarly $ \frac{1}{s_1+s_2}\int_0^{s_1+s_2} \tilde \varphi^p=y.$ 
Consequently, we have
\begin{align*}
\mathcal{B}(x,y,s,t)&\geq \int_0^{t^{p/q}} u^{q/p-1}\left(\frac{1}{u}\int_0^u \tilde\varphi\right)^q\mbox{d}u=\frac{p}{q}\int_0^t \left(\frac{1}{u^{p/q}}\int_0^{u^{p/q}} \tilde\varphi\right)^q\mbox{d}u.
\end{align*}
Now, fix $\e>0$ and consider the function $\Psi:[0,t_1]\times[0,t_2]\to \R$ given by
\begin{align*} 
\Psi(w_1,w_2)&=\int_0^{w_1+w_2} \left(\e+\frac{1}{u^{p/q}}\int_0^{u^{p/q}} \tilde\varphi\right)^q\mbox{d}u\\
&\quad 
-\int_0^{w_1} \left(\frac{1}{u^{p/q}}\int_0^{u^{p/q}} \varphi_1\right)^q\mbox{d}u
-\int_0^{w_2} \left(\frac{1}{u^{p/q}}\int_0^{u^{p/q}} \varphi_2\right)^q\mbox{d}u.
\end{align*}
Note that $\Psi$ is of class $C^1$. Furthermore, since $\tilde\varphi$ is the nonincreasing rearrangement of the concatenation of $\varphi_1$ and $\varphi_2$, the function $\Psi$ is positive on the axes: $\Psi(0,w_2)> 0$ and $\Psi(w_1,0)> 0$. Now, fix $w=(w_1,w_2)$ with $w_1$, $w_2$ strictly positive. 
Since $p\leq q$, we have $(w_1+w_2)^{p/q}\leq  w_1^{p/q}+w_2^{p/q}$ and  
\begin{align*}
 &\e+\frac{1}{(w_1+w_2)^{p/q}}\int_0^{(w_1+w_2)^{p/q}} \tilde\varphi\\
&>  \frac{1}{w_1^{p/q}+w_2^{p/q}}\int_0^{w_1^{p/q}+w_2^{p/q}} \tilde\varphi\\
 &\geq \frac{1}{w_1^{p/q}+w_2^{p/q}}\left(\int_0^{w_1^{p/q}} \varphi_1+\int_0^{w_2^{p/q}} \varphi_2\right)\\
 &= \frac{w_1^{p/q}}{w_1^{p/q}+w_2^{p/q}}\cdot \frac{1}{w_1^{p/q}}\int_0^{w_1^{p/q}} \varphi_1+\frac{w_2^{p/q}}{w_1^{p/q}+w_2^{p/q}}\cdot \frac{1}{w_2^{p/q}}\int_0^{w_2^{p/q}} \varphi_2.
 \end{align*}
The second inequality  above is due to the fact that $\tilde\varphi$ is the rearrangement of the concatenation of $\varphi_1$ and $\varphi_2$. Therefore, we conclude that
\begin{equation}\label{strict}
\begin{split}
& \e+\frac{1}{(w_1+w_2)^{p/q}}\int_0^{(w_1+w_2)^{p/q}} \tilde\varphi\geq  \min\left\{\frac{1}{w_1^{p/q}}\int_0^{w_1^{p/q}} \varphi_1,\, \frac{1}{w_2^{p/q}}\int_0^{w_2^{p/q}} \varphi_2\right\}.
\end{split}
\end{equation}
Thus, we have $\Psi_{w_1}(w_1,w_2)> 0$ or $\Psi_{w_2}(w_1,w_2)> 0$. Since $\Psi$ is positive on the axes, this easily implies that $\Psi $ is positive on its full domain. Letting $\e\to 0$, we obtain
\begin{align*}
 \mathcal{B}(x,y,s,t)&\geq \frac{p}{q}\int_0^{t_1+t_2} \left(\frac{1}{u^{p/q}}\int_0^{u^{p/q}} \tilde\varphi\right)^q\mbox{d}u\\
&\geq 
\frac{p}{q}\int_0^{t_1} \left(\frac{1}{u^{p/q}}\int_0^{u^{p/q}} \varphi_1\right)^q\mbox{d}u
+\frac{p}{q}\int_0^{t_2} \left(\frac{1}{u^{p/q}}\int_0^{u^{p/q}} \varphi_2\right)^q\mbox{d}u.
\end{align*}
Taking the supremum over $\varphi_1$ and $\varphi_2$ as above gives the desired property 3$^\circ$''.
\end{proof}

\section{Proof of Theorem \ref{mainthm}}

\subsection{Proof of \eqref{mainin}} Let $\sigma=v^{1-p'}$ be the dual weight to $v$. In our considerations below, we will use the notation 
$$ \langle \psi\rangle_{Q,\sigma}=\frac{1}{\sigma(Q)}\int_Q \psi \mbox{d}\sigma$$
for the associated weighted average. A central role in the proof of our main estimate is played by the following sharp fractional version of the Carleson imbedding theorem. The constant $C_{p,q}$ appearing below is as in \eqref{defC}.

\begin{theorem}\label{change_of_measure}
Suppose that $(a_Q)_{Q\in \mathcal{T}}$ is a sequence of nonnegative numbers satisfying the Carleson condition
\begin{equation}\label{Carle}
\left(\sum_{Q'\in \mathcal{T}(Q)} a_{Q'}\right)^{1/q}\leq \sigma(Q)^{1/p}\qquad \mbox{for all }Q\in \mathcal{T}.
\end{equation}
Then for any nonnegative function $\varphi$ on $X$ we have
\begin{equation}\label{assertion}
\left(\sum_{Q\in \mathcal{T}} a_Q \langle \varphi\rangle_{Q,\sigma}^q \right)^{1/q}\leq C_{p,q}\left(\int_X \varphi^p\mbox{d}\sigma\right)^{1/p}.
\end{equation}
The constant is the best possible.
\end{theorem}
\begin{proof}
We will exploit the abstract Bellman function $\mathcal{B}$ introduced in the previous section. Fix a sequence $(a_Q)_{Q\in \mathcal{T}}$ and a function $\varphi$ as in the statement of the theorem. Consider the sequences $(\varphi_n)_{n\geq 0}$, $(\psi_n)_{n\geq 0}$, $(\xi_n)_{n\geq 0}$ and $(\eta_n)_{n\geq 0}$ of functions on $X$, given as follows. For $\omega\in X$, set
$$ \varphi_n(\omega)=\langle \varphi\rangle_{Q_n(\omega),\sigma},\qquad \psi_n(\omega)=\langle \varphi^p\rangle_{Q_n(\omega),\sigma},\qquad \xi_n(\omega)=\sigma(Q_n(\omega))$$
and
$$ \eta_n(\omega)=\sum_{Q'\in \mathcal{T}(Q_n(\omega))} a_{Q'},$$
where $Q_n(\omega)$ is the unique element of $\mathcal{T}^{(n)}$ which contains $\omega$. 

\smallskip

Fix $n$ and an arbitrary element $Q$ of $\mathcal{T}^{(n)}$. Then $\varphi_n$, $\psi_n$, $\xi_n$ and $\eta_n$ are constant on $Q$. Denote  their corresponding values by $x$, $y$, $s$ and $t$. Next, let $Q_1$, $Q_2$, $\ldots$, $Q_m$ be the children of $Q$ belonging to $\mathcal{T}^{(n+1)}$.  That is, the sets $Q_1$, $Q_2$, $\ldots$, $Q_m$ are pairwise disjoint, belong to $\mathcal{T}^{(n+1)}$ and their union is $Q$. Then the functions $\varphi_{n+1}$, $\psi_{n+1}$, $\xi_{n+1}$ and $\eta_{n+1}$ are constant on each $Q_j$.  Denote their corresponding values by $(x_j,y_j,s_j,t_j):=(\varphi_{n+1},\psi_{n+1},\xi_{n+1},\eta_{n+1})|_{Q_j}$. It is easy to see that the conditions \eqref{conditions} are satisfied, with $u=a_{Q}$. For example,
$$ \sum_{i=1}^m s_ix_i=\sum_{i=1}^m \sigma(Q_i)\langle \varphi\rangle_{Q_i,\sigma}=\sum_{i=1}^m \int_{Q_i}\varphi\mbox{d}\sigma=\int_Q \varphi\mbox{d}\sigma=\sigma(Q)\langle \varphi\rangle_{Q}=sx$$
and the remaining identities are checked similarly. Consequently, the condition 3$^\circ$ of Theorem \ref{PropB} implies
$$ \mathcal{B}(x,y,s,t)\geq \frac{p}{q}x^qu+\sum_{i=1}^m \mathcal{B}(x_i,y_i,s_i,t_i),$$
which is equivalent to
$$ \mathcal{B}(\varphi_n,\psi_n,\xi_n,\eta_n)|_Q\geq \frac{p}{q}a_Q\langle \varphi\rangle_Q^q+\sum_{i=1}^m \mathcal{B}(\varphi_{n+1},\psi_{n+1},\xi_{n+1},\eta_{n+1})|_{Q_i}.$$
Consequently, summing over all $Q\in \mathcal{T}^{(n)}$, we obtain
\begin{align*}
 &\sum_{Q\in \mathcal{T}^{(n)}}\mathcal{B}(\varphi_n,\psi_n,\xi_n,\eta_n)|_Q\\
&\qquad \qquad \geq \frac{p}{q}\sum_{Q\in \mathcal{T}^{(n)}}a_Q\langle \varphi\rangle_Q^q+\sum_{Q\in \mathcal{T}^{(n+1)}}\mathcal{B}(\varphi_{n+1},\psi_{n+1},\xi_{n+1},\eta_{n+1})|_Q.
\end{align*}
Therefore, by induction and the property 1$^\circ$ of Theorem \ref{PropB},
\begin{align*}
\mathcal{B}(\varphi_0,\psi_0,\xi_0,\eta_0)|_X
&=\sum_{Q\in \mathcal{T}^{(0)}}\mathcal{B}(\varphi_n,\psi_n,\xi_n,\eta_n)|_Q\\
&\geq \frac{p}{q}\,\sum_{n=0}^{N-1} \sum_{Q\in \mathcal{T}^{(n)}}a_Q \langle \varphi\rangle_Q^q+\sum_{Q\in \mathcal{T}^{(N)}}\mathcal{B}(\varphi_{N},\psi_{N},\xi_{N},\eta_{N})|_Q\\
&\geq \frac{p}{q}\,\sum_{n=0}^{N-1} \sum_{Q\in \mathcal{T}^{(n)}}a_Q \langle \varphi\rangle_Q^q. 
\end{align*}
Letting $N\to \infty$ we arrive at
$$ \mathcal{B}(\varphi_0,\psi_0,\xi_0,\eta_0)|_X\geq \frac{p}{q}\,\sum_{Q\in \mathcal{T}}a_Q \langle \varphi\rangle_Q^q,$$
by Lebesgue's monotone convergence theorem. Finally, we exploit the condition 2$^\circ$ to obtain
\begin{align*}
\mathcal{B}(\varphi_0,\psi_0,\xi_0,\eta_0)|_X&=\mathcal{B}\left(\langle \varphi\rangle_{X,\sigma},\langle \varphi^p\rangle_{X,\sigma},\sigma(X),\sum_{Q\in \mathcal{T}}a_Q\right)\leq \frac{p}{q}C_{p,q}^q \left(\int_X \varphi^p\mbox{d}\sigma\right)^{q/p},
\end{align*}
which combined with the preceding estimate yields the assertion.
\end{proof}

We are now ready for the proof of our main result.

\begin{proof}[Proof of \eqref{mainin}]
It is enough to  establish the estimate
\begin{equation}\label{wewant}
 \left(\int_X (\mathcal{M}^\alpha_\mathcal{T} g)^qu\right)^{1/q}\leq C_{p,q}\left(\int_X g^pv\right)^{1/p}
\end{equation}
for $\mathcal{T}$-simple functions $g$,  that is, for linear combinations of characteristic functions of elements of $\mathcal{T}$. Let us  first linearize the fractional maximal operator. Since $g$ is simple, for each $\omega$ there is $\tilde{Q}=\tilde{Q}(\omega)\in \mathcal{T}$ such that $\mathcal{M}^\alpha_\mathcal{T} g=\mu(\tilde Q)^\alpha \langle g\rangle_{\tilde Q}$. Such a $\tilde Q$ need not be unique: if this is the case, we pick $\tilde Q\in \mathcal{T}^{(n)}$ with $n$ as small as possible. Now, for a given $Q\in \mathcal{T}$, let
$$ E(Q)=\{\omega\in X:\tilde Q(\omega)=Q\}.$$
By the construction, we see that $E(Q)\subset Q$ and the sets $\{E(Q)\}_{Q\in \mathcal{T}}$ are pairwise disjoint. The aforementioned linearization of $\mathcal{M}^\alpha_\mathcal{T}$ reads
$$ \mathcal{M}^\alpha_\mathcal{T} g=\sum_{Q\in \mathcal{T}} \mu(Q)^\alpha \langle g\rangle_Q \chi_{E(Q)}.$$
Thanks to this identity, the powers of $\mathcal{M}^\alpha_\mathcal{T} g$ are handled easily: since $\{E(Q)\}_{Q\in \mathcal{T}}$ are pairwise disjoint, we have
$$ \int_X (\mathcal{M}^\alpha_\mathcal{T} g)^qu=\sum_{Q\in \mathcal{T}} (\mu(Q)^\alpha \langle g\rangle_{Q})^q u(E(Q)).$$ 
Now, setting $f=g\sigma^{-1}$, we see that 
$$ \langle g\rangle_{Q}=\frac{1}{\mu(Q)}\int_Q g\sigma^{-1}\sigma \mbox{d}\mu=\frac{\sigma(Q)}{\mu(Q)}\cdot \frac{1}{\sigma(Q)}\int_Q f\mbox{d}\sigma=\langle \sigma\rangle_{Q}\langle f\rangle_{Q,\sigma}.$$
Therefore, the left-hand side of \eqref{wewant} equals
$$ \left(\sum_{Q\in \mathcal{T}} \langle f\rangle^q_{Q,\sigma}\cdot \big(\mu(Q)^\alpha \langle \sigma \rangle_Q\big)^qu(E(Q))\right)^{1/q}=\left(\sum_{Q\in \mathcal{T}} a_Q \langle f\rangle^q_{Q,\sigma}\right)^{1/q},$$
where we have set $a_Q=\big(\mu(Q)^\alpha \langle \sigma \rangle_Q\big)^qu(E(Q))$; the right-hand side of \eqref{wewant} is
$$ C_{p,q}\left(\int_X g^pv\right)^{1/p}=C_{p,q}\left(\int_X f^p\sigma\right)^{1/p}.$$
Thus, our claim is precisely the estimate \eqref{assertion}, and hence all we need is to check whether the sequence $\{a_Q\}_{Q\in \mathcal{T}}$ satisfies the Carleson condition \eqref{Carle}. But this is precisely the testing condition: indeed, for any $Q\in \mathcal{T}$, we have, by the very definition of the fractional maximal operator,
\begin{align*}
 \left(\sum_{Q'\in \mathcal{T}(Q)} a_{Q'}\right)^{1/q}&=\left(\sum_{Q'\in \mathcal{T}(Q)}\int_{E(Q')} (\mu(Q')^\alpha \langle \sigma\rangle_{Q'})^q \mbox{d}u\right)^{1/q}\\
&\leq \left(\int_Q \mathcal{M}^\alpha_\mathcal{T} (\sigma\chi_Q)^qu\right)^{1/q}.
\end{align*}
The latter quantity is bounded from above by $(\sigma(Q))^{1/p}$, in the light of \eqref{testing}. This gives the claim.
\end{proof}

\subsection{Sharpness} Now we will prove that for $\alpha\geq 1/p-1/q$ the constant in \eqref{mainin} cannot be improved. To this end, suppose that the estimate \eqref{mainin} holds with some constant $\tilde{c}_{p,q}$. 
Consider the probability space $(X,\mathcal{F},\mu)$ equal to the interval $[0,1]$ equipped with its Borel subsets and Lebesgue's measure. Fix a large positive integer $N$ and consider the tree structure $\mathcal{T}$ given as follows: for $0\leq n\leq N$, the family $\mathcal{T}^{(n)}$ consists of the interval $\left[0,\frac{N-n}{N}\right]$ and the intervals $(\frac{k}{N},\frac{k+1}{N}]$, $k=N-n$, $N-n+1$, $\ldots$, $N-1$. The remaining families (i.e., $\mathcal{T}^{(n)}$ for $n>N$), are taken to be arbitrary families of intervals. Introduce the weights 
$$ u(\omega)=\frac{(1-\alpha)q}{p}\omega^{q(1-\alpha)/p-1},\qquad v(\omega)=(1-\alpha)^{1-p}\omega^{\alpha(p-1)},$$
so that the dual weight $\sigma=v^{1-p'}$ is given by $\sigma(\omega)=(1-\alpha)\omega^{-\alpha}$. Let us check that $u$ and $v$ satisfy the testing condition. We will need the following technical fact.

\begin{lemma}\label{technical}
For any $x>1$ and $s>0$ we have
$$ \frac{x^{1/s}+1}{s(x^{1/s}-1)}\ln x\geq 2.$$
\end{lemma}
\begin{proof}
The claim is equivalent to 
$$ D(x):=(x^{1/s}+1)\ln x-2s(x^{1/s}-1)\geq 0.$$
If we let $x\downarrow 1$, then both sides become equal; thus it is enough to prove that $D'(x)>0$ for any $x>1$. We compute directly that 
$$ D'(x)=\frac{1}{x}\left(\frac{1}{s}x^{1/s}\ln x+1-x^{1/s}\right).$$
Denoting the expression in the parentheses by $E(x)$, we see that $\lim_{x\downarrow 1}E(x)=0$ and  $ E'(x)=s^{-2}x^{1/s-1}\ln x\geq 0$. This implies $E\geq 0$, $D'\geq 0$ and finally, $D\geq 0$. The proof is complete.
\end{proof}

\begin{lemma}
The weights $u$, $v$ satisfy the testing condition \eqref{testing}.
\end{lemma}
\begin{proof}
Fix $Q\in \mathcal{T}$. We consider two cases. If $Q=[0,r]$ for some $r>0$, then for any $\omega\in [a,b]\subset Q$ we have
$$ \mu([a,b])^\alpha \langle \sigma \rangle_{[a,b]}\leq \mu([0,b])^\alpha \langle \sigma  \rangle_{[0,b]}=b^\alpha \cdot b^{-\alpha}=1,$$
since $\sigma$ is decreasing on $[0,1]$. This gives $\mathcal{M}^\alpha_\mathcal{T} (\sigma\chi_Q)(\omega)=1$ and therefore,
\begin{align*}
 \left(\int_Q (\mathcal{M}^\alpha_\mathcal{T} (\sigma\chi_Q))^q \mbox{d}u\right)^{1/q}&=\left(\int_0^r u(\omega)\mbox{d}\omega\right)^{1/q}
=r^{(1-\alpha)/p}=\left(\int_0^r \sigma(\omega)\mbox{d}\omega\right)^{1/p},
\end{align*}
as desired. Now we consider the case in which $Q=[\ell,r]$ with $\ell\neq 0$. Again we start with the formula for the fractional maximal function of $\sigma\chi_Q$. For an arbitrary interval $[a,b]$ with $\mu([a,b]\cap Q)>0$, we have
$$ \mu([a,b])^\alpha \langle \sigma\chi_Q\rangle_{[a,b]}\leq \mu([a,b]\cap Q)^\alpha \langle \sigma\rangle_{[a,b]\cap Q}.$$
 Next, for any interval $[a,b]$ contained within $Q$ (that is, satisfying $\ell \leq a\leq b\leq r$), we consider  the quantity
$$ F(a,b):=\mu([a,b])^\alpha \langle \sigma \rangle_{[a,b]}=(b-a)^{\alpha-1}(b^{1-\alpha}-a^{1-\alpha}).$$
Arguing as above, we claim  that $F(a,b)\leq F(\ell,b)$.  Indeed, $F(a,b)$ is the product of $(b-a)^\alpha$ and $\langle \sigma\rangle_{[a,b]}$, and both factors are decreasing functions of $a$ since  $\sigma$ is decreasing. Furthermore, we compute that
$$ F_b(a,b)=(1-\alpha)a(b-a)^{\alpha-2}\big(a^{-\alpha}-b^{-\alpha}\big)\geq 0,$$
and hence $F(a,b)\leq F(\ell, b)$, as claimed. 

 Putting all the above facts together and noting that the elements of $\mathcal{T}$ are intervals, we obtain $\mathcal{M}^\alpha(\sigma\chi_Q)(\omega)=F(\ell,r).$ Consequently, \eqref{testing} will be established if we show that
$$ F(\ell,r)\left(\int_\ell^r u(\omega)\mbox{d}\omega\right)^{1/q}\leq \left(\int_\ell^r \sigma(\omega)\mbox{d}\omega\right)^{1/p},$$
which is equivalent to
$$ (r-\ell)^{\alpha-1}\left(r^{1-\alpha}-\ell^{1-\alpha}\right)^{1-1/p}\left(r^{q(1-\alpha)/p}-\ell^{q(1-\alpha)/p}\right)^{1/q}\leq 1.$$
If $\ell=0$, then both sides are equal: we have proved this in the previous case. Thus, it is enough to show that the left-hand side, considered as a function of $\ell \in [0,r]$, is nonincreasing. Denoting the left-hand side by $L(\ell)$, we check that $L'(\ell)\leq 0$ is equivalent to
$$ \frac{1}{r-\ell}\leq \frac{(1-\frac{1}{p})\ell^{-\alpha}}{r^{1-\alpha}-\ell^{1-\alpha}}+\frac{\frac{1}{p}a^{q(1-\alpha)/p-1}}{r^{q(1-\alpha)/p}-\ell^{q(1-\alpha)/p}},
$$
or, after the substitution $x=r/\ell\geq 1$, that 
$$ \frac{1}{x-1}\leq \frac{1-1/p}{x^{1-\alpha}-1}+\frac{1/p}{x^{q(1-\alpha)/p}-1}.$$
Note that it is enough to check this estimate for $\alpha=1/p-1/q$, since the right-hand side is an increasing function of $\alpha$. However, this boundary case is   Jensen's inequality for the function $G(s)= (x^{1/s}-1)^{-1}$, $s>0$, since
$$ \left(1-\frac{1}{p}\right)\cdot \frac{1}{1-\alpha}+\frac{1}{p}\cdot \frac{p}{q(1-\alpha)}=\frac{1-\frac{1}{p}+\frac{1}{q}}{1-\alpha}=1.$$ 
Thus we will be done if we prove that $G''(s)\geq 0$. We check that 
$$ G''(s)=(x^{1/s}-1)^{-2}s^{-3}\left(\frac{x^{1/s}+1}{s(x^{1/s}-1)}\ln x-2\right),$$
and it suffices to make use of Lemma \ref{technical}.
\end{proof}

We return to the sharpness. 
Since $u$, $v$ satisfy the testing condition, the inequality \eqref{mainin} implies
$$ \left(\int_0^1 \big(\mathcal{M}^\alpha_\mathcal{T} \varphi\big)^q \mbox{d}u\right)^{1/q}\leq \tilde{c}_{p,q}\left(\int_0^1 f^p\mbox{d}v\right)^{1/p}.$$
Now, any $\omega\in [0,1]$ is contained in the interval $\big[0,\lceil N\omega\rceil/N\big]$, which belongs to $\mathcal{T}$. Consequently, by the very definition of the fractional maximal function,
$$ \mathcal{M}^\alpha_\mathcal{T} \varphi\geq \mu\big(\big[0,\lceil N\omega\rceil /N\big]\big)^{\alpha-1}\int_0^{\lceil N\omega\rceil/N} \varphi\mbox{d}\mu.$$
Combining this bound with the previous inequality and letting $N\to \infty$ gives
\begin{align*}
 &\left(\int_0^1 \left(t^{\alpha-1}\int_0^t \varphi\right)^q \frac{(1-\alpha)q}{p}t^{q(1-\alpha)/p-1}\mbox{d}t\right)^{1/q}\\
&\qquad \qquad\qquad \qquad  \qquad \leq \tilde{c}_{p,q}\left(\int_0^1 \varphi^p(1-\alpha)^{p-1}t^{\alpha(p-1)}\mbox{d}t\right)^{1/p}.
\end{align*}
Let $\beta=1/(1-\alpha)$. Substituting $t=s^\beta$ and letting $f(r)=\varphi(r^\beta)r^{\beta-1}$, the above inequality is equivalent to 
$$ \left(\int_0^1 u^{q/p-1}\left(\frac{1}{u}\int_0^u f\right)^q\mbox{d}u\right)^{1/q}\leq \left(\frac{p}{q}\right)^{1/q}\tilde{c}_{p,q}\left(\int_0^1 f^p\right)^{1/p},$$
which is Bliss' inequality localized to $[0,1]$. It remains to apply a dilation argument: plugging $\tilde{f}(u/T)=f(u)$ and letting $T\to \infty$, we obtain Bliss' inequality on $(0,\infty)$, with the same constant $\left(p/q\right)^{1/q}\tilde{c}_{p,q}$. Consequently, we must have $\tilde{c}_{p,q}\geq C_{p,q}$, and this is precisely the desired sharpness.

\end{document}